\documentclass[12pt]{amsart}

\usepackage{amssymb}

\input xy

\usepackage{xy}\xyoption{all}

\newtheorem{theorem}{Theorem}[section]

\newtheorem{lemma}[theorem]{Lemma}
\newtheorem{prop}[theorem]{Proposition}

\newtheorem{cor}[theorem]{Corollary}
\newtheorem*{prop*}{Proposition}
\newtheorem*{lemma*}{Lemma}
\newtheorem*{claim*}{Claim}

\theoremstyle{definition}

\newtheorem{example}[theorem]{Example}
\newtheorem{remark}[theorem]{Remark}

\theoremstyle{remark}

\numberwithin{equation}{section}

\def\cstarlambda{C^\ast(\Lambda)}
\def\cstar{C^\ast}
\def\ZL0{{\Z^{\Lambda^0}}}
\def\Lambdablue{\Lambda^{e_1}}
\def\Lambdared{\Lambda^{e_2}}

\def\C{\mathbb{C}}
\def\Z{\mathbb{Z}}
\def\N{\mathbb{N}}
\def\T{\mathbb{T}}
\def\O{\mathcal{O}}

\newcommand{\id}{\operatorname{id}}
\newcommand{\per}{\operatorname{per}}

\newcommand{\Zmod}[1]{\Z/#1\Z}
\newcommand{\trace}{\operatorname{trace}}

\newcommand{\Ind}{\operatorname{Ind}}
\newcommand{\Aut}{\operatorname{Aut}}

\newcommand{\Perm}{\operatorname{Perm}}\newcommand{\Prim}{\operatorname{Prim}}

\usepackage{pictexwd}

\newcommand{\sixdominovtx}[7]{\beginpicture \setcoordinatesystem units <#1cm,#1cm>
\setplotarea x from 0 to 6, y from 0 to 1 \setlinear \plot 0 0 6 0 /
\plot 0 0 0 1 / \plot 0 1 6 1 / \plot 1 0 1 1 / \plot 2 0 2 1 /
\plot 3 0 3 1 / \plot 4 0 4 1 / \plot 5 0 5 1 / \plot 6 0 6 1 /
\put{#2} at 0.5 0.5 \put{#3} at 1.5 0.5 \put{#4} at 2.5 0.5 \put{#5}
at 3.5 0.5 \put{#6} at 4.5 0.5 \put{#7} at 5.5 0.5 \endpicture}

\newcommand{\twodomino}[3]{\beginpicture \setcoordinatesystem units <#1cm,#1cm>
\setplotarea x from 0 to 2, y from 0 to 1 \setlinear \plot 0 0 2 0 /
\plot 0 1 2 1 / \plot 0 0 0 1 / \plot 1 0 1 1 / \plot 2 0 2 1 /
\put{#2} at 0.5 0.5 \put{#3} at 1.5 0.5 \endpicture}

\newcommand{\threedomino}[4]{\beginpicture \setcoordinatesystem units <#1cm,#1cm>
\setplotarea x from 0 to 3, y from 0 to 1 \setlinear \plot 0 0 3 0 /
\plot 0 1 3 1 / \plot 0 0 0 1 / \plot 1 0 1 1 / \plot 2 0 2 1 /
\plot 3 0 3 1 / \put{#2} at 0.5 0.5 \put{#3} at 1.5 0.5 \put{#4} at
2.5 0.5
\endpicture}

\begin{document}
\title[Periodic $2$-graphs]{\boldmath{Periodic $2$-graphs arising from subshifts}}
\author[David Pask]{David Pask}
\address{David Pask, School  of Mathematics and Applied Statistics, University of Wollongong, NSW 2522, Australia}
\email{dpask@uow.edu.au}
\author[Iain Raeburn]{Iain~Raeburn}
\address{Iain Raeburn, School  of Mathematics and Applied Statistics, University of Wollongong, NSW 2522, Australia}
\email{raeburn@uow.edu.au}
\author[Natasha A. Weaver]{Natasha A. Weaver}
\address{Natasha A. Weaver, School of Mathematical and Physical Sciences,
University of Newcastle, NSW 2308, Australia}
\email{Tash.Weaver@gmail.com}

\thanks{This research was supported by the Australian Research Council, and Natasha Weaver was supported by an Australian Postgraduate
Award.}

\begin{abstract}
Higher-rank graphs were introduced by Kumjian and Pask to provide models for higher-rank Cuntz-Krieger algebras. In a previous paper, we constructed $2$-graphs whose path spaces are rank-two subshifts of finite type, and showed that this construction yields aperiodic $2$-graphs whose $C^*$-algebras are simple and are not ordinary graph algebras.
Here we show that the construction  also gives a family of periodic $2$-graphs which we call \emph{domino
graphs}. We investigate the combinatorial structure of domino graphs, finding interesting points of contact with the existing combinatorial literature, and prove a structure theorem for the $C^*$-algebras of domino graphs.
\end{abstract}
\date{\today}

\maketitle

%\newpage

\section{Introduction}

Higher-rank graphs (or $k$-graphs) are combinatorial objects which are higher-di\-men\-sional analogues of directed graphs. They were  invented by Kumjian and Pask \cite{KP} to provide combinatorial models for a family of higher-rank Cuntz-Krieger algebras studied by Robertson and Steger \cite{RS2}. We now know that many important $C^*$-algebras can be realised as the $C^*$-algebras of higher-rank graphs, and there is a good deal of interest in understanding different classes of higher-rank graphs (see, for example, \cite{PRRS,DPY,SZ}).

Here we are interested primarily in $2$-graphs. Intuitively, a $2$-graph is a directed graph $\Lambda:=(\Lambda^0,\Lambda^1,r,s)$ in which the set $\Lambda^1$ is partitioned into a set $\Lambdablue$ of blue edges and set $\Lambdared$ of red edges, together with a set $C$ of \emph{commuting squares}
\begin{equation}
\label{samplesq}\xygraph{{\bullet}="v11":@{-->}[d]{\bullet}="v10"^h:[l]{\bullet}="v00"^g"v11":[l]{\bullet}="v01"_f:@{-->}"v00"_e}
\end{equation}
in which every blue-red path $gh$ and every red-blue path $ef$ occur exactly once. We view $C$ as defining an equivalence relation on the path space $\Lambda^*$ which allows each path with $i$ blue edges and $j$ red edges to be rewritten in any chosen order of blue and red edges. If the square \eqref{samplesq} belongs to $C$, for example, then $ef=gh$ as paths in $\Lambda$.

In \cite{PRW}, we constructed a family of $2$-graphs whose infinite path spaces are rank-two subshifts of finite type, as studied by Schmidt \cite{Schmidt}. We found conditions which ensure that the $C^*$-algebras of these $2$-graphs are simple and purely infinite, and used results of Evans \cite{E} to compute their $K$-theory in a wide range of examples. Here we consider a family of $2$-graphs which we call \emph{domino graphs}. They are also built using the construction in \cite{PRW}, but their $C^*$-algebras are definitely not simple. We find interesting connections with known combinatorial objects, such as necklaces and Lyndon words, and we analyse the structure of their $C^*$-algebras. 

Our analysis uses two main operator-algebraic tools. We realise domino graphs as crossed products of an ordinary directed graph (or more strictly speaking, of the $1$-graph it defines) by an action of $\Z$, as studied in recent work of Farthing, Pask and Sims \cite{FPS}; the results in \cite{FPS} then imply that the $C^*$-algebra of a domino graph is the crossed product of an ordinary graph algebra by an action of $\Z$. Next, we observe that this action of $\Z$ on the graph algebra has large isotropy, and use a theorem of Olesen and Pedersen \cite{OP} to realise the crossed product as an induced $C^*$-algebra associated to a simple crossed product. This part of the  analysis may be of some independent interest: we provide a new version of the Olesen-Pedersen theorem which gives a very specific isomorphism and which is easier to apply.

We begin in \S\ref{s_2graphs} with a quick review of $k$-graphs, emphasising the connection to the intuitive description of $2$-graphs given above. Then in  \S\ref{sec-cps}, we discuss the general facts about crossed-product $C^*$-algebras which we need, and describe how they apply to the crossed-product graphs of \cite{FPS}.  In \S\ref{s_dominographs}, we review the construction of \cite{PRW}, as it applies to dominos, and then analyse the structure of the resulting domino graphs using ideas from combinatorics which we discuss in an appendix. The main result (Proposition~\ref{prop_lambdaisacrossedproduct}) says that most domino graphs are crossed product $2$-graphs of the sort studied in \cite{FPS}. In \S\ref{S_c*-algebras}, we combine the results of the preceding two sections to prove our structure theorem for the $C^*$-algebras of domino graphs (Theorem~\ref{thm_cstaralg}). In the last section, we compute the $K$-theory of domino-graph algebras.

\section{$k$-graphs and their $C^*$-algebras}\label{s_2graphs}

A \emph{$k$-graph} is a pair $(\Lambda,d)$ consisting of a countable
category $\Lambda$ and a functor $d:\Lambda\to\mathbb{N}^k$, called
the \emph{degree map}, satisfying the \emph{factorisation property}:
for every $\lambda\in\Lambda$ and $m,n\in\N^k$ with
$d(\lambda)=m+n$, there exist unique elements $\mu,\nu\in\Lambda$
such that $d(\mu)=m$, $d(\nu)=n$ and $\lambda$ is the composition
$\mu\nu$. We write $\Lambda^0$ for the set of objects, $\Lambda$ for
the set of morphisms, and $s,r:\Lambda\to \Lambda^0$ for the domain
and codomain maps, so that $\lambda$ and $\mu$ are composable
exactly when $s(\lambda)=r(\mu)$. The factorisation property implies
that for every $v\in \Lambda^0$,  there is a unique morphism $\mu$
such that $s(\mu)=r(\mu)$ and $d(\mu)=0$, namely the identity
morphism $\iota_v$ at $v$. We use $v\mapsto \iota_v$ to identify
$\Lambda^0$ with a subset of $\Lambda$. For $n\in \N^k$, we define
$\Lambda^n:=\{\lambda\in \Lambda:d(\lambda)=n\}$, and refer to an
element $\lambda$ of $\Lambda^n$ as a \emph{path of degree $n$ from
$s(\lambda)$ to $r(\lambda)$}.

In this paper, we are concerned primarily with $2$-graphs, and not much will be missed by a reader who assumes $k=2$ throughout. We visualise a $2$-graph $\Lambda$ using its \emph{skeleton}, which is the directed bicoloured graph
with vertex set $\Lambda^0$ and edge set $\Lambdablue\cup\Lambdared$, where the elements $\beta$ of
$\Lambda^{e_1}$ are blue edges from $s(\beta)\in
\Lambda^0$ to $r(\beta)\in \Lambda^0$, and the elements of
$\Lambda^{e_2}$ are red edges. (In print, black curves 
represent blue edges and dashed curves represent red edges.) The skeleton thus consists of two directed graphs $B\Lambda:=(\Lambda^0,\Lambdablue,r,s)$ and $R\Lambda:=(\Lambda^0,\Lambdared,r,s)$ with the same vertex set and different coloured edges.

The factorisation property in a $2$-graph $\Lambda$ is completely determined by the factorisations of paths of degree $(1,1)$: writing $(1,1)=e_1+e_2$ and $(1,1)=e_2+e_1$ gives
a bijection between the blue-red paths of length $2$ and the
red-blue paths of length $2$. We visualise a path of degree
$(1,1)$ as a commuting square
\begin{equation}
\label{exsquare}
\xygraph{{\bullet}="v11":@{-->}[d]{\bullet}="v10"^h:[l]{\bullet}="v00"^g
"v11":[l]{\bullet}="v01"_f:@{-->}"v00"_e}
\end{equation}
in which the bijection matches up the blue-red path $gh$ with the
red-blue path $ef$, so that $gh=ef$ are the two factorisations of
the path of degree $(1,1)$. Then the factorisation property is determined by a family $C$ of commuting squares in which each red-blue and each blue-red path occurs exactly once (see \cite[\S6]{KP}). We then view paths of degree $(3,2)$ from $w$ to $v$, for example, as
copies of the rectangle
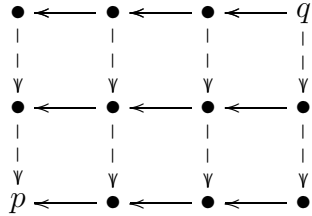
\begin{equation*}
\centerline{
\xygraph{
{q}="v43":[l]{\bullet}="v33":[l]{\bullet}="v23":[l]{\bullet}="v13"
"v43":@{-->}[d]{\bullet}="v42":@{-->}[d]{\bullet}="v41"
"v33":@{-->}[d]{\bullet}="v32":@{-->}[d]{\bullet}="v31"
"v23":@{-->}[d]{\bullet}="v22":@{-->}[d]{\bullet}="v21"
"v13":@{-->}[d]{\bullet}="v12":@{-->}[d]{p}="v11"
"v43":"v33":"v23":"v13"
"v42":"v32":"v22":"v12"
"v41":"v31":"v21":"v11"
}}
\end{equation*}
pasted round the blue-red graph, in such a way that $q$ lands on $w$, $p$ lands
on $v$, and each constituent square belongs to $C$. Composing two paths involves adding squares from $C$ to fill out a larger rectangle, and \cite[\S6]{KP} says there is only one way to do this.

Suppose that $\Lambda$ is a $k$-graph in which every vertex receives paths of all degrees. The $C^*$-algebra of $\Lambda$ is the universal $C^*$-algebra $C^*(\Lambda)$ generated by  partial isometries $\{s_\lambda:\lambda\in \Lambda\}$ satisfying
\begin{itemize}
\item[(a)] $\{s_v:v\in \Lambda^0\}$ are mutually orthogonal projections,
\smallskip
\item[(b)] $s_\lambda s_\mu=s_{\lambda\mu}$ whenever $s(\lambda)=r(\mu)$ ,
\smallskip
\item[(c)] $s^*_\lambda s_\lambda=s_{s(\lambda)}$ for every $\lambda\in \Lambda$, and
\smallskip
\item[(d)] $s_v=\sum_{r(\lambda)=v,\ d(\lambda)=n}s_\lambda s_\lambda^*$ for every $v\in \Lambda^0$ and $n\in \N^k$.
\end{itemize}
For information about higher-rank graphs and their $C^*$-algebras, see \cite{KP} or \cite{RSY03}.

\section{Crossed products of graph algebras}\label{sec-cps}

In \cite{FPS}, Farthing, Pask and Sims consider actions $\alpha:\Z^l\to \Aut \Lambda$ of $\Z^l$ on a $k$-graph $\Lambda$ and the induced actions $\tilde \alpha$ of $\Z^l$ on $C^*(\Lambda)$ \cite[Proposition~3.1]{FPS}.  In this section, we establish some general properties of crossed products of the form $C^*(\Lambda)\times_{\tilde\alpha}\Z^l$.

If $\alpha:G\to \Aut A$ is an action of a group on a $C^*$-algebra, then the crossed product $A\times_\alpha G$ is generated by a universal covariant representation $(i_A, i_G)$ in $M(A\times_\alpha G)$. We write $\pi\times U$ for the representation of $A\times_\alpha G$ corresponding to a covariant representation $(\pi,U)$ of $(A,G,\alpha)$.  Here we are only interested in actions of discrete groups, and then the map $i_A$ takes values in $A\times_\alpha G$, and the elements of the form $i_A(a)i_G(s)$ span a dense subspace of $A\times_\alpha G$. When $G$ is abelian, the crossed product carries a canonical dual action $\hat\alpha$ of the dual group $\hat G$, which is characterised\footnote{There is disagreement in the literature about the definition of the dual action: in \cite{W}, for example, $i_G(s)$ would be multiplied by $\overline{\xi(s)}=\xi(s^{-1})$. It does not make a big difference, since $s\mapsto s^{-1}$ is an automorphism of the abelian group $G$, and induces an isomorphism of the two crossed products.} by $\hat\alpha_\xi(i_A(a)i_G(s))=\xi(s)i_A(a)i_G(s)$.

To state our result about crossed products of the form $C^*(\Lambda)\times _{\tilde\alpha}\Z^l$, we need one further concept. Suppose that $L$ is a closed
subgroup of a compact group $K$, and $\alpha:L\to\Aut A$ is a
continuous action. Then
\[\Ind_L^K(A,\alpha)=\{f\in C(K,A):f(gh)=\alpha_h^{-1}(f(g)) \text{ for }g\in K,h\in L\}.\]
is a $C^*$-subalgebra of $C(K,A)$, called an \emph{induced $C^*$-algebra}. (The extra assumption in  \cite[\S6.3]{RW}
that the function $sL\mapsto\|f(s)\|$ vanishes at $\infty$ on $K/L$ is automatic here because $K$ is compact).

\begin{theorem}\label{cpofkgraph}
Suppose that $\alpha$ is an action of $\Z^l$ on a finite $k$-graph
$\Lambda$. Then there are a cofinite
subgroup $N$ of $\Z^l$ and an injective homomorphism $\beta:\Z^l/N\to \Aut\Lambda$ such that $\alpha_m=\beta_{m+N}$. Let $\Phi:C^*(\Lambda)\times_{\tilde\alpha}\Z^l\to \cstar(\Lambda)\times_{\tilde{\beta}}(\Z^l/N)$ be the homomorphism such that $\Phi\circ i_{C^*(\Lambda)}=i_{C^*(\Lambda)}$ and $\Phi(i_{\Z^l}(m))=i_{\Z^l/N}(m+N)$. Then there is an isomorphism $\Psi$ of $C^*(\Lambda)\times_{\tilde\alpha}\Z^l$ onto $\Ind_{N^\perp}^{\T^l}(\cstar(\Lambda)\times_{\tilde{\beta}}(\Z^l/N),\widehat{\tilde{\beta}})$ such that $\Psi(b)(z)=\Phi(\tilde\alpha_z^{-1}(b))$.
\end{theorem}

Once we have shown the existence of $N$ and $\beta$, we will apply the following general result of Olesen and Pedersen \cite{OP}. 

\begin{theorem}\label{OPind}
Suppose that $H$ is a subgroup of a discrete abelian group $G$,  and $\gamma$ is an action of $G/H$ on a $C^*$-algebra $A$. Let $q:G\to G/H$ be the quotient map, let $\beta:=\gamma\circ q$, and define
$\Phi:A\times_{\beta} G\to A\times_{\gamma} (G/H)$ by
$\Phi:=i_A\times(i_{G/H}\circ q)$. Then $\Psi(b)(\xi)=\Phi(\hat{\beta_\xi}^{-1}(b))$
defines an isomorphism $\Psi$of $A\times_{\beta}G$ onto
$\Ind_{H^\perp}^{\hat{G}}(A\times_\gamma (G/H), \hat{\gamma})$.
\end{theorem}

With some effort, Theorem~\ref{OPind} can be deduced from \cite[Corollary~2.5]{OP}, which in turn is deduced from a chain of results involving both the ``restricted crossed products'' of Dang Ngoc \cite{DN} and the ``twisted crossed products'' of Green \cite{G}. Since we are only interested in ordinary crossed products, we give a short direct proof, which is similar to that of \cite[Theorem~2.1]{ARR}.

\begin{proof} When we view $\xi\in H^\perp$ as a character of $G/H$, we have $\Phi\circ\hat\beta_\xi=\hat\gamma_\xi\circ\Phi$, and an easy calculation using this shows that $\Psi(b)$ belongs to the induced algebra. 

Next we show that $\Psi$ is surjective. Since each generator $i_A(a)i_{G/H}(s+H)$ is just $\Phi(i_A(a)i_G(s))$, $\Phi$ is surjective. Thus for each $\xi$, the elements $\Psi(b)(\xi)=\Phi(\hat\beta_\xi^{-1}(\hat\beta_\xi(b)))$ fill out $A\times_{\gamma} (G/H)$. Next, observe that the integrated form $i_G|_H:H\to UM(A\times_\gamma G)$ maps $C^*(H)$ into the centre $Z(M(A\times_\gamma G))$, and hence by Fourier transformation gives a central action of $C(\hat H)=C(\hat G/H^\perp)$ on $A\times_\gamma G$. The algebra $C(\hat H)=C(\hat G/H^\perp)\subset C(\hat G)$ acts by pointwise multiplication on the induced algebra, and we claim that $\Psi$ is then $C(\hat H)$-linear. To see this, consider a generator $\delta_h$ for $C^*(H)$: since the Fourier transform of $\delta_h$ is the function $\epsilon_{h^{-1}}:\xi\mapsto \overline{\xi(h)}$, we need to check that $\Psi(i_G(h)b)(\xi)=\overline{\xi(h)}\Phi(b)(\xi)$. For $b=i_A(a)i_G(s)$, we have
\begin{align*}
\Psi(i_G(h)b)(\xi)&=\Phi(\hat\beta_{\xi}^{-1}(i_G(h)i_A(a)i_G(s)))=\Phi(\hat\beta_{\xi}^{-1}(i_A(a)i_G(h+s)))\\
&=\Phi(\overline{\xi(h+s)}i_A(a)i_G(h+s))=\overline{\xi(h)\xi(s)}i_A(a)i_{G/H}(s+H)\\
&=\overline{\xi(h)}\Psi(\beta_{\xi}^{-1}(b))=\overline{\xi(h)}\Psi(b)(\xi),
\end{align*}
as required. Thus $\Psi$ is $C(\hat H)$-linear, as claimed. Now a partition-of-unity argument (as in the Lemma on page~704 of \cite{Ech}, for example) shows that $\Psi$ is surjective.

To see that $\Psi$ is injective, it suffices to show that every irreducible representation $\pi\times U$ of $A\times_\beta G$ factors through $\Psi$. 
Since $G$ is abelian, the operators $U_h$ for $h\in H$ commute with every $U_s$, and since $\beta_h$ is the identity on $A$, $U_h$ commutes with every $\pi(a)$ too. So $U$ maps $H$ into the commutant $(\pi\times U)(A\times G)'$, which is $\C1$ because $\pi\times U$ is irreducible, and there is a character $\theta$ of $H$ such that $U_h=\theta(h)1$ for $h\in H$. By Pontryagin duality, there is a character $\chi\in \hat G$ such that $\chi|_H=\theta$. Now $\chi^{-1}U$ factors through a unitary representation $V$ of $G/H$, and $(\pi,V)$ is covariant because $(\pi,U)$ is. Since $(\pi\times U)(b)=(\pi\times V)(\Psi(b)(\chi))$, the result follows. \end{proof}

\begin{proof}[Proof of Theorem~\ref{cpofkgraph}]  We begin by showing that
$\Aut\Lambda$ is finite. $\Aut \Lambda$ consists of bijective functors
$\alpha:\Lambda\to\Lambda$ which preserve the degree map:
$d(\alpha(\lambda))=d(\lambda)$ for $\lambda\in\Lambda$. In particular, for each $n\in\N^k$ $\alpha$ is a bijection
of $\Lambda^n$ onto $\Lambda^n$. Define $\phi:\Aut\Lambda\to
\bigoplus_{i=1}^k\Perm \Lambda^{e_i}$ by taking $\phi(\alpha)$ to be
the tuple $(\alpha|_{\Lambda^{e_i}})_{i=1}^k$. Then $(\alpha\beta)|_{\Lambda^{e_i}}=\alpha|_{\Lambda^{e_i}}\circ\beta|_{\Lambda^{e_i}}$, so $\phi$ is a homomorphism. If
$\phi(\alpha)=\phi(\beta)$, then $\alpha(e) = \beta(e)$ for all
edges $e\in\bigcup_{i=1}^k\Lambda^{e_i}$, and it follows from the  factorisation property that $\alpha(\lambda)=\beta(\lambda)$ for all $\lambda$; thus $\phi$ is one-to-one. Since $\Lambda$ is a finite
$k$-graph, $\bigoplus_{i=1}^k\Perm \Lambda^{e_i}$ is a finite group, and the injectivity of $\phi$ implies that $|\Aut\Lambda|<\infty$.
Let $N:=\ker\alpha$, and let $q:\Z^l\to \Z^l/N$ be the quotient map. Then there is a unique monomorphism $\beta:\Z^l/N\to \Aut\Lambda$ such that $\alpha=\beta\circ q$. Since $\Aut\Lambda$ is finite, this implies in particular that  $\Z^l/N$ is finite. The corresponding actions on $C^*(\Lambda)$ satisfy $\tilde\alpha=\tilde\beta\circ q$, so the result follows from Theorem~\ref{OPind}.
\end{proof}

\begin{cor}\label{mapcylappl}
If $\alpha$ is an automorphism of a $k$-graph $\Lambda$ such that $\alpha^n$ is the identity, then $C^*(\Lambda)\rtimes_{\tilde\alpha}\Z$ is isomorphic to the mapping cylinder
\[
M(\gamma):=\{f:[0,1]\to C^*(\Lambda)\times_{\tilde\alpha}(\Z/n\Z):\text{ $f$ is continuous and $f(1)=\gamma(f(0)$}\}
\]
associated to the generator $\gamma:=\hat{\tilde\alpha}_{e^{2\pi i/n}}$ of the dual action.
\end{cor}

\begin{proof}
The subgroup $N$ in Theorem~\ref{cpofkgraph} is $N=n\Z$, so $N^\perp$ is the group $C_n$ of $n$th roots of unity. But if $\theta:C_n\to \Aut B$ is an action on a $C^*$-algebra $B$, then the map $\phi:\Ind_{C_n}^\T(B,\theta)\to C([0,1],B)$ defined by $\phi(f)(t)=f(e^{-2\pi it/n})$ is an isomorphism onto the mapping cylinder $M(\theta_{e^{2\pi i/n}})$.
\end{proof}

The description of $C^*(\Lambda)\times_{\tilde\alpha}\Z^l$ as an induced algebra gives us, for example, a description of its primitive ideal space as the quotient of $\T^l\times \Prim (C^*(\Lambda)\times_{\tilde\beta}(Z^l/N))$ by the diagonal action of $N^\perp$ (see \cite[Proposition~6.16]{RW}). This is particularly interesting if $C^*(\Lambda)$ is simple, for then $\Prim C^*(\Lambda)=\{0\}$, and $\Prim(C^*(\Lambda)\times_{\tilde\alpha}\Z^l)$ is homeomorphic to $\T^l/N^\perp=\hat N\cong \T^l$ (because $N$, being a subgroup of a free abelian group, is itself free abelian). So we would like to know when $C^*(\Lambda)\times (\Z/n\Z)$ is simple. For $k=1$, $\Lambda$ is the path space of a directed graph $E$ and $C^*(\Lambda)=C^*(E)$, and this question has an elegant answer.

We say that a direct graph $E$ is \emph{$n$-connected} if for every pair of vertices $v,w\in E^0$, there is a path of length $n$ with source $v$ and range $w$. The following result is basically due to Katayama and Takehana \cite{KT}.

\begin{prop}\label{KTouter}
Suppose that $E$ is a finite directed graph which is $n$-connected for some $n\geq 1$. If $\alpha$ is an automorphism of $E$ which is not the identity, then $\tilde\alpha\in C^*(E)$ is outer.
\end{prop}

\begin{proof}
We consider the Hilbert bimodule $X(E)$ over $c(E^0)$ constructed in \cite{FR} and \cite[\S8]{R}: $X(E)$ is the completion of the bimodule $X_0=C(E^1)$  with
actions and inner products
\[
(a\cdot f\cdot b)(e) = a(r(e))f(e)b(s(e))\ \text{ and }\ \langle f,g\rangle_A (v)=\sum_{e\in E^1, s(e)=v}\overline{f(e)}g(e).
\]
Indeed, since the graph can have no sources or sinks, the module $X_0$ will already be complete and is full. Then $\{k_{X(E)}(\delta_e), k_A(\delta_v)\}$ is a Cuntz-Krieger $E$ family in $\O(X(E))$, and the corresponding homomorphism of $C^*(E)$ into $\O(X(E))$ is an isomorphism (see \cite[Example~8.13]{R}, for example). The elements $|s^{-1}(s(e))|^{-1/2}\delta_e$ form a Parseval frame for $X(E)$, so $X(E)$ is ``of finite type'' as required in \cite{KT}. The $E^0\times E^0$ matrix $M$ constructed in \cite[page~497]{KT} is the vertex matrix of $E$, and the hypothesis of $n$-connectedness says precisely that $M^n(v,w)>0$ for every $v,w\in E^0$. So Proposition~2.3 of \cite{KT} says that $C(E^0)$ is $X(E)$-aperiodic.

We define $\theta:C(E^0)\to C(E^0)$ by $\theta(f):=f\circ \alpha_0^{-1}$ and $U:X(E)\to X(E)$ by $U(x):=x\circ\alpha_1^{-1}$, and then $(\theta,U)$ satisfy the hypotheses (3.1) used in \cite{KT}. Since $\alpha$ is not the identity, the bijection $\alpha_1$ of $E^1$ is not the identity, and since automorphisms of $X(E)$ of the form $x\mapsto uxu^*$ are $C(E^1)$-linear, $U$ cannot have this form. So Proposition~3.8 of \cite{KT} implies that $\alpha_U$ is outer. Since $\alpha_U(\delta_e)=\delta_{\alpha_1(e)}$ and $\theta(\delta_v)=\delta_{\alpha_0(v)}$, the isomorphism of $C^*(E)$ onto $\O(X(E))$ carries $\tilde\alpha$ into $\alpha_U$, and $\tilde\alpha$ is outer too.
\end{proof}

\begin{remark}
Since Theorem~2.4 of \cite{KT} says that $X$-aperiodicity is equivalent to simplicity of the core $\O(X)^\gamma$, we should reconcile Proposition~\ref{KTouter} with the results of Pask and Rho on simplicity of $C^*(E)^\gamma$ for finite $E$ \cite{PaskRho}. (Examples in \cite{PaskRho} show that the results do not extend to infinite $E$.) Pask--Rho define the period $\per(v)$ of a vertex to be the greatest common divisor of $S(v):=\{|\mu|:s(\mu)=r(\mu)=v,\ |\mu|>1\}$; if $E$ is strongly connected, then all vertices have the same period, called the \emph{period} of $E$. They prove in \cite[Theorem~6.2]{PaskRho} that $C^*(E)^\gamma$ is simple if and only if $E$ is strongly connected with period $1$. This and Proposition~\ref{KTouter} imply that a graph $E$ is strongly connected with period $1$ if and only if $E$ is $n$-connected for some $n$. We should, of course, be able to see this directly.

If $E$ is $n$-connected, then we can get from a vertex $v$ to each of its neighbours by a path of length $n$, and hence back to $v$ by one of length $n+1$. Hence $v$ admits return paths of length $n+k$ for every $k\in \N$, and $\per(v)=1$. If $E$ is strongly connected with period $1$, then there exists $m$ such that the subsemigroup $S(v)$ of $\N$ contains $m+\N$. If $r$ is the diameter of $E$, then we can get from any vertex to any other by a path of length $m+2r$: choose a path through $v$ of length $s\leq 2r$, and insert a return path at $v$ of length $m+(2r-s)$. 
\end{remark}

\begin{cor}\label{applyKish}
Suppose that $\alpha$ is an action of $\Z^l$ on an $n$-connected directed graph $E$ which does not consist of a single loop, and $\beta$ is the injection $\beta:\Z^l/N\to \Aut E$ of a finite quotient of $\Z^l$ such that $\alpha=\beta\circ q$. Then $C^*(E)\times_{\tilde\beta}(\Z^l/N)$ is simple, and the primitive ideal space of $C^*(E)\times_{\tilde\alpha}\Z^l$ is homeomorphic to $\T^l$.
\end{cor}

\begin{proof}
Theorem~\ref{cpofkgraph} tells us that there is an action $\beta$ of a finite quotient $\Z^l/N$ such that $\beta_{m+N}$ is nontrivial whenever $m+N\not=N$. So Proposition~\ref{KTouter} imples that $\tilde\beta_{m+N}$ is outer for every non-zero $m+N$. The $n$-connectedness hypothesis implies that $E$ cannot consist of single cycle of length greater than $1$, so the simplicity theorem for graph algebras (or the original theorem of Cuntz and Krieger) tells us that $C^*(E)$ is simple. Thus Kishimoto's \cite[Theorem~3.1]{Kish} implies that $C^*(E)\times_{\tilde\beta}(\Z^l/N)$ is simple. Now the result follows from the isomorphism of Theorem~\ref{cpofkgraph} and the description of the primitive ideal space of induced algebras in \cite[Proposition~6.61]{RW}.
\end{proof}

It would be interesting to have similar criteria for the simplicity of $C^*(\Lambda)\times_{\tilde\beta}(\Z^l/N)$ when $\Lambda$ has rank greater than $1$. Robertson and Sims have given us good criteria for the simplicity of $C^*(\Lambda)$ \cite{RobS}, and of course Kishimoto's theorem is still valid, so what is missing is a version of the Katayama-Takehana theorem (that is, Proposition~\ref{KTouter}) for $k$-graphs. It is not obvious, though, that the methods used in \cite{KT} will carry over.

\section{Domino graphs}\label{s_dominographs}

The $2$-graphs in \cite{PRW} are associated to \emph{basic data}
$(T,q,t,w)$ consisting of:
\begin{itemize}
\item a \emph{tile} $T$, which is a hereditary ($j\in T$
and $0\leq i\leq j$ imply $i\in T$) subset of $\N^2$ with finite
cardinality $|T|$;
\item an \emph{alphabet} $\{0,1,\ldots,q-1\}$, where $q\geq 2$ is an integer; we view the alphabet as a commutative ring by identifying it with $\Zmod{q}$ in the obvious way;
\item an element $t$ of the alphabet, called the \emph{trace}; and
\item a function $w:T\to \{0,1,\ldots,q-1\}$, called the \emph{rule}.
\end{itemize}
Provided certain values of $w$ are invertible elements of $\Zmod{q}$, Theorem~3.4 of \cite{PRW} tells us how to construct a $2$-graph $\Lambda=\Lambda(T,q,t,w)$ from this data.

A \emph{domino} is a tile of the form
$T=\{0,e_1,2e_1,\ldots,(n-1)e_1\}$, which is determined by $|T|:=n$. In this
paper, we write $(n,q,t)$ for the basic data consisting of the
domino of cardinality $n$, alphabet $\Zmod{q}$ and trace $t$, and we
always take the rule $w$ to be the constant function $1$.
Then every value of $w$ is invertible, and \cite[Theorem~3.4]{PRW} gives a
$2$-graph $\Lambda=\Lambda(n,q,t)$, which we call a \emph{domino graph}.

Each $2$-graph $\Lambda$ of \cite{PRW} is finite, has no sources
\cite[Proposition~3.2]{PRW}, is strongly connected in the sense
that $u\Lambda v$ is nonempty for all
$u,v\in\Lambda^0$ \cite[Proposition~5.3]{PRW}, and has at most one edge
of each colour between any pair of vertices
\cite[Proposition~3.5(b)]{PRW}. So domino graphs
have all these properties.

We now briefly recall the construction of \cite[\S2]{PRW} as it applies to domino graphs.
We picture the domino $T$ with $|T|=n$ as a row of $n$ boxes; for
example, we draw the domino with $n=6$ as $\sixdominovtx{0.5}{}{}{}{}{}{}$\,, and we picture a vertex in
$\Lambda$ as a copy of $T$ in which each box is filled with an
element of the alphabet so that the sum of the entries is $t\pmod{q}$. Formally,
\[\Lambda(n,q,t)^0 =\Big\{v:T\to\Z/q\Z : \sum_{i\in T}v(i)=t\pmod{q}\Big\},\]
and in $\Lambda(6,2,0)$, for example, the
function $v:T\to\Zmod{2}$ with
$v(0)=v(3e_1)=0$, $v(e_1)=v(2e_1)=v(4e_1)=v(5e_1)=1$ gives a vertex pictured as $\sixdominovtx{0.5}{0}{1}{1}{0}{1}{1}$\,.

To describe the paths, we need some notation. For $S\subset \Z^2$
and $m\in \Z^2$, we define $S+m:=\{i+m:i\in S\}$ and $T(m):=\bigcup_{0\leq l\leq m}T+l$. If
$f:S\to\Zmod{q}$ is a function defined on a subset $S$ of $\N^2$
containing $T+l$, then we define $f|_{T+l}:T\to\Zmod{q}$ by
\begin{equation*}
f|_{T+l} (i) = f(i+l) \text{ for } i\in T.
\end{equation*}
A \emph{path of degree $m$} is a function $\lambda:T(m)\to \Zmod{q}$
such that $\lambda|_{T+l}$ is a vertex for every $0\leq l\leq m$; then
$\lambda$ has \emph{source} $s(\lambda)=\lambda|_{T+m}$ and
\emph{range} $r(\lambda)=\lambda|_T$. Then $\Lambda^m$ denotes the
set of paths of degree $m$ and $\Lambda^\ast=\bigcup_{m\geq
0}\Lambda^m$. In pictures, paths in $\Lambda^\ast$ are
block diagrams covered by translates of $T$, filled in so that each
translate is a valid vertex. For example, in  $\Lambda(6,2,0)$, the diagram
\begin{equation}\label{diag_mu}
\beginpicture \setcoordinatesystem units %%%%%%MU%%%%%%%%%%
<0.5cm,0.5cm> \setplotarea x from 0 to 9, y from 0 to 3 \setlinear
\plot 0 0 9 0 / \plot 0 1 9 1 / \plot 0 2 9 2 / \plot 0 3 9 3 /
\plot 0 0 0 3 / \plot 1 0 1 3 / \plot 2 0 2 3 / \plot 3 0 3 3 /
\plot 4 0 4 3 / \plot 5 0 5 3 / \plot 6 0 6 3 / \plot 7 0 7 3 /
\plot 8 0 8 3 / \plot 9 0 9 3 / \put{$\mu=$} at -1 1.5
%first row
\put{0} at 0.5 0.5 \put{1} at 1.5
0.5 \put{1} at 2.5 0.5 \put{0} at 3.5 0.5 \put{1} at 4.5 0.5 \put{1}
at 5.5 0.5 \put{0} at 6.5 0.5 \put{1} at 7.5 0.5 \put{1} at 8.5 0.5
%second row
\put{0} at 0.5 1.5 \put{0} at 1.5 1.5 \put{0} at 2.5 1.5 \put{1} at
3.5 1.5 \put{1} at 4.5 1.5 \put{0} at 5.5 1.5 \put{0} at 6.5 1.5
\put{0} at 7.5 1.5 \put{0} at 8.5 1.5
%third row
\put{0} at 0.5 2.5 \put{0} at 1.5 2.5 \put{0} at 2.5 2.5 \put{0} at
3.5 2.5 \put{0} at 4.5 2.5 \put{0} at 5.5 2.5 \put{0} at 6.5 2.5
\put{0} at 7.5 2.5 \put{0} at 8.5 2.5
\endpicture
\end{equation} represents a path $\mu$ of degree $(3,2)$ from
$s(\mu)=\sixdominovtx{0.5}{0}{0}{0}{0}{0}{0}$ (the upper RH
translate of $T$) to $r(\mu)=\sixdominovtx{0.5}{0}{1}{1}{0}{1}{1}$ (the
lower LH translate).

The paths $\Lambda^\ast$ form a category with the composition
defined in \cite[Proposition~3.2]{PRW}: we say $\mu\in\Lambda^m$ and
$\nu\in \Lambda^p$ are \emph{composable} if $s(\mu)=r(\nu)$ and then
there exists a unique path $\lambda\in\Lambda^{m+p}$ satisfying $\lambda|_{T(m)}=\mu$ and $\lambda|{T(n)-m}=\nu$.
So if $\mu$ is the path in \eqref{diag_mu} and $\nu$ is the path of
degree $(2,1)$ below
\[\beginpicture \setcoordinatesystem units <0.5cm,0.5cm>
\setplotarea x from 0 to 8, y from 0 to 2 \setlinear \plot 0 0 8 0 /
\plot 0 0 0 2 / \plot 0 1 8 1 / \plot 0 2 8 2 / \plot 1 0 1 2 /
\plot 2 0 2 2 / \plot 3 0 3 2 / \plot 4 0 4 2 / \plot 5 0 5 2 /
\plot 6 0 6 2 / \plot 7 0 7 2 / \plot 8 0 8 2 / \put{0} at 0.5 0.5
\put{0} at 1.5 0.5 \put{0} at 2.5 0.5 \put{0} at 3.5 0.5 \put{0} at
4.5 0.5 \put{0} at 5.5 0.5 \put{0} at 6.5 0.5 \put{0} at 7.5 0.5
\put{1} at 0.5 1.5 \put{1} at 1.5 1.5 \put{1} at 2.5 1.5 \put{1} at
3.5 1.5 \put{1} at 4.5 1.5 \put{1} at 5.5 1.5 \put{1} at 6.5 1.5
\put{1} at 7.5 1.5 \put{$\nu=$} at -1 1
\endpicture\]
then $s(\mu)=r(\nu)$, and $\mu\nu$ is found by filling in the empty squares
in the diagram
\[\beginpicture \setcoordinatesystem units %%%%%%MU%%%%%%%%%%
<0.5cm,0.5cm> \setplotarea x from 0 to 11, y from 0 to 3 \setlinear
\plot 0 0 9 0 / \plot 0 1 9 1 / \plot 0 2 9 2 / \plot 0 3 9 3 /
\plot 0 0 0 3 / \plot 1 0 1 3 / \plot 2 0 2 3 / \plot 3 0 3 3 /
\plot 4 0 4 3 / \plot 5 0 5 3 / \plot 6 0 6 3 / \plot 7 0 7 3 /
\plot 8 0 8 3 / \plot 9 0 9 3 / \put{$\mu=$} at -1 1.5
%first row
\put{0} at 0.5 0.5 \put{1} at 1.5 0.5 \put{1} at 2.5 0.5 \put{0} at
3.5 0.5 \put{1} at 4.5 0.5 \put{1} at 5.5 0.5 \put{0} at 6.5 0.5
\put{1} at 7.5 0.5 \put{1} at 8.5 0.5
%second row
\put{0} at 0.5 1.5 \put{0} at 1.5 1.5 \put{0} at 2.5 1.5 \put{1} at
3.5 1.5 \put{1} at 4.5 1.5 \put{0} at 5.5 1.5 \put{0} at 6.5 1.5
\put{0} at 7.5 1.5 \put{0} at 8.5 1.5
%third row
\put{0} at 0.5 2.5 \put{0} at 1.5 2.5 \put{0} at 2.5 2.5 \put{0} at
3.5 2.5 \put{0} at 4.5 2.5 \put{0} at 5.5 2.5 \put{0} at 6.5 2.5
\put{0} at 7.5 2.5 \put{0} at 8.5 2.5 %%%%%%%%%%%% NU
\plot 3 2 11 2 / \plot 3 2 3 4 / \plot 3 3 11 3 / \plot 3 4 11 4 /
\plot 4 2 4 4 / \plot 5 2 5 4 / \plot 6 2 6 4 / \plot 7 2 7 4 /
\plot 8 2 8 4 / \plot 9 2 9 4 / \plot 10 2 10 4 / \plot 11 2 11 4 /
\put{0} at 3.5 2.5 \put{0} at 4.5 2.5 \put{0} at 5.5 2.5 \put{0} at
6.5 2.5 \put{0} at 7.5 2.5 \put{0} at 8.5 2.5 \put{0} at 9.5 2.5
\put{0} at 10.5 2.5 \put{1} at 3.5 3.5 \put{1} at 4.5 3.5 \put{1} at
5.5 3.5 \put{1} at 6.5 3.5 \put{1} at
7.5 3.5 \put{1} at 8.5 3.5 \put{1} at 9.5 3.5 \put{1} at 10.5 3.5 %%%%%%%%
\setdashpattern <3pt,3pt> \plot 0 3 0 4 / \plot 0 4 3 4 / \plot 9 0
11 0 / \plot 11 0 11 2 / \plot 1 3 1 4 / \plot 2 3 2 4 / \plot 9 1
11 1 / \plot 10 0 10 2 /
\endpicture\;;\]
the proof of \cite[Proposition~3.2]{PRW} shows that can do this in such a way that the entry for
each empty box is completely determined by previous steps. With
$d:\Lambda\to\N^2$ defined by $d(\lambda)=m$ for
$\lambda\in\Lambda^m$, $\Lambda=\Lambda(n,q,t)$ is a $2$-graph \cite[Theorem~3.4]{PRW}.

\begin{remark}
Theorem~4.1 of \cite{PRW} says the two-sided infinite path space
$\Lambda(n,q,0)^\Delta$ is conjugate to the $2$-dimensional shift of finite
type with underlying space
\[
\Omega= \Big\{ f:\Z^2 \to\Zmod{q}:\sum_{i\in T} f|_{T+l}(i)=
0\pmod{q} \text{ for all } l\in\Z^2 \Big\}.
\]
The shift $\Omega$ corresponds to the model $R^q_2/(g)$ of
\cite[\S3]{KitS4} in which $g$ is the cyclotomic polynomial
$g_T(u_1,u_2)=1+u_1+\cdots+u_1^{n-1}$. Theorem~6.5(2) of
\cite{Schmidt} implies that $\Omega$ is a non-mixing shift, whereas
the shifts associated to the graphs in \cite[\S5--6]{PRW} are mixing.
\end{remark}

We can explicitly describe the skeletons of domino graphs as follows.

\begin{prop}\label{prop_dominographstructure}
Suppose that $\Lambda=\Lambda(n,q,t)$ is a domino $2$-graph. Then:

\textnormal{(1)} The red graph $R\Lambda$ is
the complete directed graph $K_{q^{n-1}}$, so that
$|v\Lambdared u|=1$ for all $u,v\in\Lambda^0$.

\textnormal{(2)} $\Lambda^0$ can be identified with the set $A^n_t$ of words of length $n$
with trace $t\pmod{q}$ over the alphabet $A=\Zmod{q}$.

\textnormal{(3)} The blue graph $B\Lambda$ consists of disjoint cycles whose lengths divide $n$, and correspond to the necklaces of $A^n_t$ discussed in Appendix~\ref{s_necklaces}. For each divisor $d$ of $n$, the number $h_d$ of blue cycles of length $d$ in $\Lambda$ is
\[
h_d=\sum_{{s\in\Zmod{q},\;d^{-1}sn
\equiv t\pmod{q}}}L_q(d,s),
\]
where $L_q(d,s)$ is the number of Lyndon words of $A^{d}_s$
described in \eqref{eqn_L}.
\end{prop}

\begin{proof}Proposition~3.5 of \cite{PRW} says that $|\Lambda^0|=q^{n-1}$ and
$|v\Lambdared|=|\Lambdared v|=q^{n-1}$ for all $v\in\Lambda^0$. Since $|v\Lambdared u|\leq 1$ for all $u,v\in\Lambda^0$, this forces $|v\Lambdared u|=1$ for all $u,v\in\Lambda^0$. For
(2), we identify each vertex $v\in\Lambda^0$ with its image written
as the concatenation $v(0)v(e_1)v(2e_1)\cdots v((n-1)e_1)$, which is
an element of the set $A^n_t$. For (3), $B\Lambda$ must consist of
disjoint cycles since $|v\Lambdablue|=|\Lambdablue v|=1$ for all
$v\in\Lambda^0$ by \cite[Proposition~3.5]{PRW}. The equivalence
classes of $A^n_t$ under the rotation $\rho$ defined in 
Appendix~\ref{s_necklaces} are called necklaces. Under the
identification $\Lambda^0=A^n_t$, the cycles in $B\Lambda$
correspond to the necklaces of $A^n_t$, and the length of each cycle
equals the period of the corresponding necklace which must divide
$n$. Suppose $[a]$ is a necklace with period $d$ and let $b$ be its
Lyndon subword. Then since $t=\trace(a)=\trace(b^{n/d})$, we have
$n/d \times \trace(b)=t$. So the number of $d$-cycles in $B\Lambda$
equals the number of Lyndon words of length $d$ whose trace
$s\in\Zmod{q}$ satisfies $d^{-1}sn\equiv t\pmod{q}$.
\end{proof}

\begin{example}\label{ex_6domino}
Suppose $\Lambda=\Lambda(6,2,0)$. The divisors of $6$ are $d_1=1$,
$d_2=2$, $d_3=3$, $d_4=6$. The constants of
Proposition~\ref{prop_dominographstructure}(3) are $h_1=2$, $h_2=0$,
$h_3=2$, $h_4=4$, and so $B\Lambda$ consists of two $1$-cycles, two
$3$-cycles and four $6$-cycles. Identifying the vertex set of
$\Lambda(6,2,0)$ with the set $A^6_0$ of words of length $6$ over
$A=\{0,1\}$ with trace $0\pmod{2}$, the cycles in $B\Lambda$
correspond to the necklaces of $A^6_0$ which are listed in the first
column of Table~\ref{table_6necklaces} in Appendix~\ref{s_necklaces}. %There are two necklaces with
%period $1$, two with period $3$ and four with period $6$.
For example, the $3$-cycle in the diagram below corresponds to the
necklace $[011011]$ with period $3$ of \eqref{diag_necklace}.
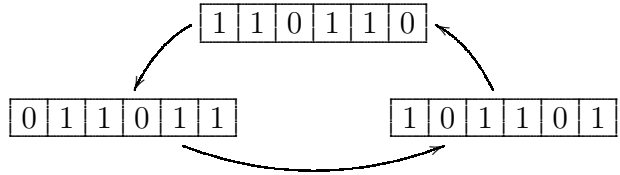
\[ \xygraph{{\sixdominovtx{0.5}{1}{1}{0}{1}{1}{0}}="v1":@/_{20pt}/[dll]{\sixdominovtx{0.5}{0}{1}{1}{0}{1}{1}}="v2":@/_{20pt}/[rrrr]{\sixdominovtx{0.5}{1}{0}{1}{1}{0}{1}}="v3"
"v3":@/_{20pt}/"v1" }\]
\end{example}

\begin{lemma}\label{lemma_sigma=id}
Suppose $\Lambda=\Lambda(n,q,t)$ is a domino $2$-graph. Define $\sigma:\Lambda^0\to\Lambda^0$ by
\begin{equation}\label{eqn_sigma}
\sigma(v)=s(e) \text{ where } e\in v\Lambdablue.
\end{equation} Then $\sigma$ is a bijection. Further, $\sigma$ is the identity if $n=1$ or
$(n,q,t)=(2,2,0)$, and otherwise $\sigma$ has order $n$.
\end{lemma}

\begin{proof}
Under the identification $\Lambda^0=A^n_t$ of
Proposition~\ref{prop_dominographstructure}(2), $\sigma$ corresponds
to $\rho$ defined in Appendix~\ref{s_necklaces}. There is a unique blue
edge entering and leaving each vertex by
Proposition~\ref{prop_dominographstructure}(3), and so $\sigma$ is a
bijection.

The order of a vertex $v\in\Lambda^0$ under $\sigma$ is the
length of the cycle on which $v$ lies. The blue graph of $\Lambda(1,q,t)$ has only one vertex
$\beginpicture \setcoordinatesystem units <0.5cm,0.5cm> \setplotarea
x from 0 to 1, y from 0 to 1 \setlinear \plot 0 0 1 0 / \plot 0 0 0
1 / \plot 1 0 1 1 / \plot 0 1 1 1 / \put{$t$} at 0.5 0.5
\endpicture$ and one blue edge $\twodomino{0.5}{$t$}{$t$}$, which is a $1$-cycle, so $\sigma=\id$. The blue graph of $\Lambda(2,2,0)$ has two
vertices $\twodomino{0.5}{0}{0}$ and $\twodomino{0.5}{1}{1}$, and
the only blue edges $\threedomino{0.5}{0}{0}{0}$ and
$\threedomino{0.5}{1}{1}{1}$ are $1$-cycles, so $\sigma=\id$.

Otherwise, note that $\sigma$ has order $n$ if and only if
$B\Lambda$ contains an $n$-cycle. If $t\neq 0$, then the vertex
$v$ with $v(0)=t$ and $v(ie_1)=0$ for $1\leq i\leq n-1$ lies on an
$n$-cycle. If $t=0$, then the vertex $u$ with $u(0)=u(e_1)=1$, and
$u(ie_1)=0$ for $2\leq i\leq n-1$ lies on an $n$-cycle.
\end{proof}

The permutation $\sigma$ satisfies $\sigma(r(e))=s(e)$, and hence moves against the direction of the edges. For
example, in $\Lambda(6,2,0)$, $\sigma$ takes $\sixdominovtx{0.5}{0}{1}{1}{0}{1}{1}$ to $\sixdominovtx{0.5}{1}{1}{0}{1}{1}{0}$.

Let $E$ and $F$ be $1$-graphs. The product graph
$(E\times F,d)$ is the $2$-graph consisting
of the product category $E\times F$ with the degree
map $d(\lambda_1,\lambda_2)=(d(\lambda_1),d(\lambda_2))$. The following lemma says that the domino graphs for which $\sigma$ is
the identity are the only ones which are product graphs.

\begin{lemma}\label{lemma_productgraphs}
The domino graph $\Lambda(n,q,t)$ is a product graph if and
only if $n=1$, in which case $\Lambda\cong
K_1\times K_1$, or $(n,q,t)=(2,2,0)$, in which case $\Lambda\cong K_1\times K_2$.
\end{lemma}

\begin{proof}
Suppose $\Lambda$ is a product graph $E\times F$. Then
\[\Lambdared=(E\times F)^{e_2}=\{(v,f):v\in E^0,f\in F^1\}\]
with $s(v,f)=(v,s(f))$ and $r(v,f)=(v,r(f))$. Since $R\Lambda$ is connected
(Proposition~\ref{prop_dominographstructure}), we must have
$|E^0|=1$. Then $|F^0|=q^{n-1}$ since
\[|E^0|\times |F^0|=|E^0\times F^0|=|(E\times F)^0|=|\Lambda^0|=q^{n-1},\]
and so $B\Lambda$ has $q^{n-1}$ copies of (the underlying directed graph of) $E$.

Recall from \cite[Proposition~3.5(b)]{PRW} that there is at most one
edge of each colour between any pair of vertices in the skeleton of
$\Lambda$. So since $E$ has only one vertex, there is either no edge
or exactly one edge, which has to be a loop. If $E$ had no edge, then $B\Lambda$
would have no edges, so $E$ must be $E=K_1$.
Since $B\Lambda$ has $q^{n-1}$ copies of $E=K_1$, there is a blue loop at
each vertex. Thus $\sigma=\id$ and Lemma~\ref{lemma_sigma=id}
implies that either $n=1$ or $(n,q,t)=(2,2,0)$. If $n=1$, then
Proposition~\ref{prop_dominographstructure} implies that $h_1=1$ and
$F=K_1$; if $(n,q,t)=(2,2,0)$, then $h_1=2$ and $F=K_2$.
\end{proof}

Suppose $\alpha$ is an action of $\Z^l$ on a $k$-graph $\Lambda$.
Farthing, Pask and Sims \cite{FPS} constructed a \emph{crossed product
$(k+l)$-graph} $\Lambda\times_\alpha \Z^l$, whose underlying set is the cartesian product $\Lambda\times\N$, and which has degree map\footnote{This is slightly different from the definition in \cite{FPS}, where the degree map is defined by $d(\lambda,m):=(d(\lambda),m)$. The change has the effect of repainting the red edges blue and vice-versa; we have changed to ensure that  the isomorphism of Proposition~\ref{prop_lambdaisacrossedproduct} matches red edges with red edges.} defined by $d(\lambda,m):=(m,d(\lambda))$, range and
source maps defined by $r(\lambda,m):=(r(\lambda),0)$ and
$s(\lambda,m):=(\alpha^{-m}(s(\lambda)),0)$, and composition defined by
$(\mu,m)(\nu,n)=(\mu\alpha^m(\nu),m+n)$ when
$s(\mu,m)=r(\nu,n)$. One of the main theorems in \cite{FPS} says that 
the graph algebra $C^*(\Lambda\times _\alpha\Z^l)$ is isomorphic (in a very concrete way) to the $C^*$-algebraic crossed product $C^*(\Lambda)\times _{\tilde\alpha}\Z^l$ which we considered in \S\ref{sec-cps} \cite[Theorem~3.5]{FPS}. In our application of these ideas, we go the other way: we recognise that our domino graphs are crossed products of the form $R\Lambda\times \Z$, and then we use properties of crossed product $C^*$-algebras to study the $C^*$-algebras of domino graphs.
The action $\alpha$ is defined using the permutation $\sigma$ introduced in Lemma~\ref{lemma_sigma=id}.

\begin{prop}\label{prop_lambdaisacrossedproduct}
Suppose that $\Lambda=\Lambda(n,q,t)$ is the domino graph associated
to basic data $(n,q,t)$ with $n\geq 2$. For each $e\in
\Lambda^{e_2}$, there is a unique edge $\sigma_1(e)$ in $\Lambdared$
from $\sigma(s(e))$ to $\sigma(r(e))$, and then
$\sigma_1:\Lambdared\to\Lambdared$ is a bijection. The pair
$(\sigma,\sigma_1)$ is an automorphism of $R\Lambda$. Let $\alpha$
be the action of $\Z$ on $R\Lambda$ generated by
$(\sigma^{-1},\sigma_1^{-1})$. Then $\Lambda$ is isomorphic to the crossed
product $R\Lambda\times_\alpha\Z$.
\end{prop}

\begin{proof}
Since $R\Lambda$ is complete (by
Proposition~\ref{prop_dominographstructure}), there is exactly one
red edge between every pair of vertices in $\Lambda^0$, so
$\sigma_1$ is well-defined. To see that $\sigma_1$ is surjective,
let $f\in \Lambdared$. Then there exist unique blue edges $f_1,f_2$
with $s(f_1)=r(f)$ and $s(f_2)=s(f)$; take $e$ to be the unique edge
from $r(f_2)$ to $r(f_1)$, and then we have
$\sigma(r(e))=s(f_1)=r(f)$, $\sigma(s(e))=s(f_2)=s(f)$ and
$\sigma_1(e)=f$. To see that $\sigma_1$ is injective, suppose that
$\sigma_1(e)=\sigma_1(h)$. Then
$s(e)=\sigma^{-1}(s(\sigma_1(e)))=\sigma^{-1}(s(\sigma_1(h)))=s(h)$
and $r(e)=r(h)$, and since there is exactly one red edge between two
given vertices, we must have $e=h$. The pair $(\sigma,\sigma_1)$ is
an automorphism of $R\Lambda$ since $\sigma$ and $\sigma_1$ are
bijections and we have $s(\sigma_1(e))=\sigma(s(e))$ and
$r(\sigma_1(e))=\sigma(r(e))$ by definition. 

We build a coloured graph isomorphism $\phi$ from the skeleton of
$\Lambda$ to the skeleton of $R\Lambda\times_\alpha\Z$ and prove
that it preserves commuting squares. Then by \cite[\S6]{KP}, $\phi$
extends uniquely to a $2$-graph isomorphism, and the result follows.

We define
\begin{equation}\label{defiso}
\begin{cases}\phi_0:\Lambda^0\to(R\Lambda\times_\alpha\Z)^0\  \text{ by }\ \phi_0(v)=(v,0) \text{ for } v\in \Lambda^0,\\
\phi_1:\Lambda^{e_1}\to(R\Lambda\times_\alpha\Z)^{e_1}\ \text{ by }\ \phi_1(\beta)=(r(\beta),1) \text{ for } \beta\in\Lambdablue, \text{ and }\\
\phi_2:\Lambda^{e_2}\to(R\Lambda\times_\alpha\Z)^{e_2}\ \text{ by }\
\phi_2(\rho)=(\rho,0) \text{ for } \rho\in\Lambdared.
\end{cases}
\end{equation}
Then $\phi_0$ and $\phi_2$ are bijections because
$(R\Lambda\times_\alpha\Z)^0=\Lambda^0\times\{0\}$ and
$(R\Lambda\times_{\alpha}\Z)^{e_2}=\Lambdared\times\{0\}$. To see
that $\phi_1$ is a bijection, note that
\[(R\Lambda\times_\alpha\Z)^{e_1}=\{(r(\beta),1):\beta\in\Lambdablue\}=\{(v,1):v\in\Lambda^0\}\]
since $|v\Lambdablue|=1$ for all $v\in\Lambda^0$, and let
$\beta\in\Lambdablue$. Then $\phi_1(\beta)=(r(\beta),1)$ is the
unique edge with range $r(r(\beta),1)=(r(\beta),0)$ and source
$s(r(\beta),1)=(\alpha^{-1}(r(\beta)),0)=(s(\beta),0)$. Hence $\phi_1$ is
a bijection. So $\phi=(\phi_0,\phi_1,\phi_2)$ is an isomorphism and
it remains to show that it preserves commuting squares.

Every commuting square $\lambda\in\Lambda^{(1,1)}$ is uniquely
determined by the red edge $\lambda(0,e_2)=\lambda|_{T(e_2)}$. To
see this, let $\rho\in\Lambdared$. There are unique blue edges
$\beta_1$ and $\beta_2$ with $r(\beta_1)=r(\rho)$ and
$s(\beta_2)=s(\rho)$, and we then have
$s(\beta_1)=\alpha^{-1}(r(\beta_1))$ and $s(\beta_2)=\alpha^{-1}(r(\beta_2))$.
There is a unique red edge from $s(\beta_2)$ to $s(\beta_1)$, and it
must be $\alpha^{-1}(\rho)$ since
\begin{align*}
r(\alpha^{-1}(\rho))&=\alpha^{-1}(r(\rho))=\alpha^{-1}(r(\beta_1))=s(\beta_1), \text{ and}\\
s(\alpha^{-1}(\rho))&=\alpha^{-1}(s(\rho))=\alpha^{-1}(r(\beta_2))=s(\beta_2).
\end{align*}
So each $\rho\in\Lambdared$ determines a commuting square
$\rho\beta_2=\beta_1\alpha^{-1}(\rho)$ and every commuting square
$\lambda$ in $\Lambda$ has this form with $\rho:=\lambda|_{T(e_2)}$.

Every commuting square in $(R\Lambda\times_\alpha\Z)^{(1,1)}$ has
the form $(\rho,1)$ for some $\rho\in \Lambdared$ and has
factorisations
\[(\rho,0)(s(\rho),1)=(r(\rho),1)(\alpha^{-1}(\rho),0).\]

We will show that $\phi$ maps the commuting square
$\lambda\in\Lambda^{(1,1)}$ to the commuting square
$(\lambda|_{T(e_2)},1)$ in $R\Lambda\times_\alpha\Z$. Let
$\rho:=\lambda|_{T(e_2)}$. In pictures,
\[ \xygraph{\bullet="b":@{-}|@{>}[l] \bullet="a"_{\beta_2}
"b":@{--}|@{>}[d]\bullet="c"^{\alpha^{-1}(\rho)}
"a":@{--}|@{>}[d]\bullet="z"_{\rho} "c":@{-}|@{>}"z"^{\beta_1} }
\beginpicture \setcoordinatesystem units <1cm,1cm>
\setplotarea x from -0.5 to 1.5, y from 0 to 0 \setlinear
\put{$\overset{\phi}{\longrightarrow}$} at 0.5 -0.5
\endpicture
\xygraph{\bullet="b":@{-}|@{>}[l] \bullet="a"_{(s(\rho),1)}
"b":@{--}|@{>}[d]\bullet="c"^{(\alpha^{-1}(\rho),0)}
"a":@{--}|@{>}[d]\bullet="z"_{(\rho,0)}
"c":@{-}|@{>}"z"^{(r(\rho),1)} }\] We have $\phi_2(\rho)=(\rho,0)$
and $\phi_2(\alpha^{-1}(\rho))=(\alpha^{-1}(\rho),0)$ by definition, and
$\phi_1(\beta_2)=(r(\beta_2),1)=(s(\rho),1)$ and
$\phi_1(\beta_1)=(r(\beta_1),1)=(r(\rho),1)$. So $\phi(\lambda)$ has
factorisations $(\rho,0)(s(\rho),1)=(r(\rho),1)(\alpha^{-1}(\rho),0)$
which gives $\phi(\lambda)=(\rho,1)=(\lambda|_{T(e_2)},1)$.
\end{proof}

\section{The $\cstar$-algebras of domino graphs} \label{S_c*-algebras}

We now use what we know about the combinatorics of dominoes  to describe the $C^*$-algebra of a domino graph $\Lambda(n,q,t)$. If $n=1$ or $(n,q,t)=(2,2,0)$, then $\Lambda(n,q,t)$ is a product $2$-graph $E\times F$, and $C^*(\Lambda(n,q,t))$ is isomorphic to $C^*(E)\otimes C^*(F)$. For $n=1$, both $E$ and $F$ consist of a single loop, both $C^*(E)$ and $C^*(F)$ are isomorphic to $C(\T)$, and $C^*(\Lambda(1,q,t))$ is isomorphic to $C(\T)\otimes C(\T)=C(\T^2)$. For $(n,q,t)=(2,2,0)$, $E$ is isomorphic to $K_1$ and $F$ to $K_2$, and  $C^*(\Lambda(2,2,0))$ is isomorphic to $C(\T)\otimes C^*(K_2)=C(\T)\otimes \O_2$.

The remaining cases are handled by the following theorem. 
\begin{theorem}\label{thm_cstaralg}
Suppose that $(n,q,t)$ is  a set of basic data with $n\geq 2$ and $(n,q,t)\not=(2,2,0)$, and let $\alpha$ be the action of $\Z$ on $R\Lambda$ described in Proposition~\ref{prop_lambdaisacrossedproduct}. Let $\O_{\Lambda^0}:=C^*(t_v:v\in\Lambda^0)$ be the Cuntz algebra of the finite set $\Lambda^0=\Lambda^0(n,q,t)$, and let $\gamma$ be the automorphism of $\O_{\Lambda^0}$ such that $\gamma(t_v)=t_{\alpha_1(v)}$. Then $\gamma$ has order $n$, hence induces an action $\gamma$ of $Z/n\Z$ on $\O_{\Lambda^0}$, and there is an isomorphism $\Theta$ of $C^*(\Lambda(n,q,t))$ onto the mapping cylinder
\[
M(\hat\gamma_{e^{2\pi i/n}})=\big\{f\in C\big([0,1],\O_{\Lambda^0}\times_\gamma(\Z/n\Z)\big):f(1)=\hat\gamma_{e^{2\pi i/n}}(f(0))\big\}
\]
such that
\begin{align}
\Theta(s_v)(t)&=i_{\O_{\Lambda^0}}(t_vt_v^*)\quad\text{for $v\in\Lambda^0$,}\label{Theta1}\\
\Theta(s_e)(t)&=e^{2\pi it/n}i_{\O_{\Lambda^0}}(t_{r(e)}t_{r(e)}^*)i_{\Z/n\Z}(1+n\Z)\quad\text{for $e\in\Lambda^{e_1}$, and}\label{Theta2}\\
\Theta(s_f)(t)&=i_{\O_{\Lambda^0}}(t_{r(f)}t_{s(f)}t_{s(f)}^*)\quad\text{for $f\in\Lambda^{e_2}$.}\label{Theta3}
\end{align}
The crossed product $\O_{\Lambda^0}\times_\gamma(\Z/n\Z)$ is simple and the primitive ideal space of $C^*(\Lambda(n,q,t))$ is homeomorphic to $\T$.
\end{theorem}

\begin{proof}%[Proof of Theorem~\ref{thm_cstaralg}]
Lemma~\ref{lemma_sigma=id} implies that the permutation $\sigma^{-1}=\alpha_1|_{\Lambda^0}$ has order $n$, and hence so does the induced automorphism $\gamma$ of $\O_{\Lambda^0}$. In Proposition~\ref{prop_lambdaisacrossedproduct} (and specifically in Equation~\eqref{defiso}\,) we constructed an isomorphism $\phi$ of $\Lambda(n,q,t)$ onto the crossed product $R\Lambda\times_\alpha \Z$, and this induces an isomorphism
\begin{equation}\label{theta1}
\theta_1:C^*(\Lambda)\to C^*(R\Lambda\times_\alpha\Z).
\end{equation} 
Proposition~3.1 of \cite{FPS} implies that the
action $\alpha$ of $\Z$ on $R\Lambda$ induces an action
$\tilde{\alpha}$ of $\Z$ on the $\cstar$-algebra $\cstar(R\Lambda)$
such that $\tilde{\alpha}_m(s_\lambda)=s_{\alpha_m(\lambda)}$ for
$\lambda\in\Lambda$ and $m\in \Z$, and Theorem~3.5 of \cite{FPS} gives us an isomorphism \begin{equation}\label{theta2}
\theta_2:C^*(R\Lambda\times_\alpha\Z)\to C^*(R\Lambda)\times_{\tilde\alpha}\Z,
\end{equation}
which is characterised by $\theta_2(s_{(\mu,m)})=i_A(s_\mu)i_\Z(m)$. Since the generator is induced by a permutation of order $n$, $\tilde\alpha$ factors through an an action $\beta$ of $\Z/n\Z$ such that $\beta_{m+n\Z}(s_e)=s_{\alpha_m(e)}$, and the Olesen-Pedersen theorem in the form of Corollary~\ref{mapcylappl} gives us an isomorphism 
\begin{equation}\label{theta3}
\theta_3:C^*(R\Lambda)\times_{\tilde\alpha}\Z\to M(\hat\beta_{e^{2\pi i/n}}),
\end{equation} 
where $\hat\beta$ is the dual action of $C_n$ on the crossed product $C^*(R\Lambda)\times_{\beta}(\Z/n\Z)$. 

Next we recall from Proposition~\ref {prop_dominographstructure} that $R\Lambda$ is the complete directed graph $K_{\Lambda^0}$ with vertex set $\Lambda^0$, which is the dual of the graph $E_{\Lambda^0}$ with one vertex and the edges parametrised by $\Lambda^0$. Write $\{t_v:v\in \Lambda^0\}$ for a universal Cuntz-Krieger family which generates $\O_{\Lambda^0}:=C^*(E_{\Lambda^0})$. Then we can deduce from \cite[Corollary~2.6]{R}, for example, that there is an isomorphism $\psi$ of $C^*(R\Lambda)$ onto $\O_{\Lambda^0}$ which carries the projections $p_v$ into $t_vt_v^*$ and the partial isometries $s_e$ into $t_{r(e)}t_{s(e)}t_{s(e)}^*$. The inverse $\psi^{-1}$ takes $t_v$ to $T_v:=\sum_{r(e)=v}s_e$, and hence we have
\[\beta(\psi^{-1}(t_v))=\sum_{r(e)=v}\beta(s_e)=\sum_{r(e)=v}s_{\alpha(e)}=\sum_{r(f)=\alpha_1(v)}s_f=T_{\alpha_1(v)}=\psi^{-1}(\gamma(t_v)).
\]
Thus $\psi$ induces an isomorphism $\psi\times \id$ of $C^*(R\Lambda)\times_{\beta}(\Z/n\Z)$ onto $\O_{\Lambda^0}\times_{\gamma}(\Z/n\Z)$, and composing functions with $\psi\times \id$ gives an isomorphism $\theta_4$ of $M(\hat\beta_{e^{2\pi i/n}})$ onto $M(\hat\gamma_{e^{2\pi i/n}})$. Since $\beta$ is induced by a transitive permutation of the edges of $E_{\Lambda^0}$, we know from Corollary~\ref{applyKish} that $\O_{\Lambda^0}\times_\beta (\Z/n\Z)$ is simple and that the primitive ideal space is $\T$.

At this stage, we have an isomorphism $\Theta:=\theta_4\circ\theta_3\circ\theta_2\circ\theta_1$ of $C^*(\Lambda)$ onto $M(\hat\gamma_{e^{2\pi i/n}})$, and we need to check that $\Theta$ does the right thing on generators. For $v\in \Lambda^0$, we have
\[
\Theta(s_v)=\theta_4\circ\theta_3\circ\theta_2(s_{(v,0)})=\theta_4\circ\theta_3(i_{C^*(R\Lambda)}(p_v)),
\] 
and hence for $t\in [0,1]$ we have
\begin{align}
\label{calcTheta1}\Theta(s_v)(t)
&=\psi\times\id\big(\theta_3(i_{C^*(R\Lambda)}(p_v))\big)(t)\\
&=\psi\times\id\big(\Psi(i_{C^*(R\Lambda)}(p_v))(e^{-2\pi it/n})\big)\notag\\
&=\psi\times\id\big((i_{C^*(R\Lambda)}\times(i_{\Z}\circ q))(\hat\delta_{e^{2\pi it/n}}(i_{C^*(R\Lambda)}(p_v))\big),\notag
\end{align}
where $\delta:=\beta\circ q$ is the action of $\Z$ inflated from $\beta:\Z/n\Z\to \Aut C^*(R\Lambda)$. Since the dual action fixes the range of $i_{C^*(R\Lambda)}$, we have
\[
\label{calcTheta}\Theta(p_v)(t)
=\psi\times\id\big((i_{C^*(R\Lambda)}\times(i_{\Z}\circ q))(i_{C^*(R\Lambda)}(p_v))\big)=i_{\O_{\Lambda^0}}(t_vt_v^*),
\]
which is \eqref{Theta1}. For $f\in \Lambda^{e_2}$, we have $\Theta(s_f)=\theta_4\circ\theta_3(i_{C^*(R\Lambda)}(s_f))$, and a calculation just like \eqref{calcTheta1} gives \eqref{Theta3}. Finally, for $e\in \Lambda^{e_1}$, we have
\[
\Theta(s_e)=\theta_4\circ\theta_3\circ\theta_2(s_{(r(e),1)})=\theta_4\circ\theta_3(i_{C^*(R\Lambda)}(p_{r(e)})i_\Z(1)),
\] 
and hence for $t\in [0,1]$ we have
\begin{align*}
\Theta(s_e)&=\psi\times\id\big((i_{C^*(R\Lambda)}\times(i_{\Z}\circ q))(\hat\delta_{e^{2\pi it/n}}(i_{C^*(R\Lambda)}(p_{r(e)})i_\Z(1)))\big)\\
&=\psi\times\id\big((i_{C^*(R\Lambda)}\times(i_{\Z}\circ q))(i_{C^*(R\Lambda)}(p_{r(e)})e^{2\pi it/n}i_\Z(1))\big)\\
&=\psi\times\id\big(i_{C^*(R\Lambda)}(p_{r(e)})e^{2\pi it/n}i_{\Z/n\Z}(1+n\Z)\big)\\
&=e^{2\pi it/n}i_{\O_{\Lambda^0}}(t_{r(e)}t_{r(e)}^*)i_{\Z/n\Z}(1+n\Z).\qedhere
\end{align*}
\end{proof}

\begin{remark}
It is intriguing that the other family of periodic $2$-graphs whose algebras have been analysed also have $C^*$-algebras with primitive ideal space $\T$ (see \cite[\S5]{DY}). The graphs in \cite{DY} have just one vertex, and the skeleton admits many possible families $C$ of commuting squares; for domino graphs, the skeleton admits a unique family $C$ of commuting squares.
\end{remark}

\begin{remark}
Since the group $\Z/n\Z$ is finite, Theorem~4.1 of \cite{MR} implies that the action $\gamma$ of $\Z/n\Z$ on $\O_{\Lambda^0}$ in Theorem~\ref{thm_cstaralg} is proper in the sense of Rieffel \cite{proper}; then, since $\O_{\Lambda^0}\times_\gamma (\Z/n\Z)$ is simple, Corollary~1.7 of \cite{proper} implies that $\O_{\Lambda^0}\times_\gamma (\Z/n\Z)$ is Morita equivalent to the fixed-point algebra $\O_{\Lambda^0}^\gamma$. However, the underlying action on $R\Lambda$ is not free on any $R\Lambda^n$, so the discussion in \cite[\S4]{MR} suggests that it may be hard to get useful information about this fixed-point algebra (examples there show that it need not be the $C^*$-algebra of the quotient graph, for example).
\end{remark}

\section{$K$-theory}\label{s_ktheory}

In \S7 of \cite{PRW} we conjectured, based on the numerical evidence in
\cite[Table~1]{PRW}, that
$K_0(\cstarlambda)$ and $K_1(\cstarlambda)$ are cyclic groups of the
same order. In this section we verify this for domino graphs using
the identification of $\cstarlambda$ as a crossed product.

\begin{prop}\label{prop_dominoktheory}
If $(n,q,t)$ is  basic data, then for $i=0$ and $i=1$ we have
\[K_i(C^*(\Lambda(n,q,t)))=\begin{cases}
\Z^2 & \text{ if } n=1 \\
 0 & \text{ if } (n,q,t)=(2,2,0) \\
\Zmod{(q^{n-1}-1)} & \text{ otherwise. }
\end{cases}\]
\end{prop}

\begin{proof}
In the first case $\cstarlambda=C(\T)\otimes C(\T)$ by
Theorem~\ref{thm_cstaralg}. Since $K_i(C(\T))=\Z$
\cite[page~234]{RLL}, the $K$-groups of $C(\T)$ are torsion-free and
the K\"{u}nneth formula \cite{RoScho} gives
\begin{align*}
K_0(C(\T)\otimes C(\T))&= K_0(C(\T))\otimes K_0(C(\T))\oplus
K_1(C(\T))\otimes K_1(C(\T))\\
&= \Z\otimes\Z\oplus\Z\otimes\Z =\Z^2\\
K_1(C(\T)\otimes C(\T))&=K_0(C(\T))\otimes K_1(C(\T))\oplus
K_1(C(\T))\otimes K_0(C(\T)) \\
&= \Z\otimes\Z\oplus\Z\otimes\Z =\Z^2.
\end{align*}
In the second case $\cstarlambda=C(\T)\otimes\O_2$ by
Theorem~\ref{thm_cstaralg}, and since $K_i(\O_2)=0$
\cite[page~234]{RLL} the K\"{u}nneth formula gives $K_i(C(\T)\otimes
\O_2)=0$.
Otherwise, we have $\cstarlambda\cong \O_{\Lambda^0}\times_{\tilde\alpha}\Z$. Since $\O_{\Lambda^0}$ is a Cuntz algebra with $q^{n-1}$ generators, we deduce from \cite[page~234]{RLL}, for example, that
\[
K_1(\O_{\Lambda^0})=0 \text{ and } K_0(\O_{\Lambda^0})=\Z/(q^{n-1}-1)\Z,\]
where $K_0(\O_{\Lambda^0})$ is generated by the class $[1]$ of the identity.
Thus the Pimsner-Voiculescu sequence \cite{PV2} for this crossed product reduces to
\[ 0\longrightarrow K_1(\O_{\Lambda^0}\times_{\tilde\alpha}\Z) \longrightarrow K_0(\O_{\Lambda^0}) \overset{\id-\tilde\alpha_\ast}{\longrightarrow} K_0(\O_{\Lambda^0}) \longrightarrow K_0(\O_{\Lambda^0}\times_{\tilde\alpha} \Z) \longrightarrow 0, \]
and we have $\id-\tilde\alpha_\ast=0$ because $\tilde\alpha_\ast(1)=1$. So
both $K_0(\cstarlambda))=K_0(\O_{\Lambda^0}\times_{\tilde\alpha}\Z)$ and $K_1(\cstarlambda)=K_1(\O_{\Lambda^0}\times_{\tilde\alpha}\Z)$ are isomorphic to $\Zmod{(q^{n-1}-1)}$.
\end{proof}

\begin{appendix}
\section{Necklaces and Lyndon words}\label{s_necklaces}

A \emph{word} is a finite or infinite sequence of symbols from a
finite set $A$ called the \emph{alphabet}. Subsets of the set
$A^\ast$ of all finite words are called \emph{languages} and we
write $A^n$ for the words of length $n$. The product of two words
$u$ and $v$ is their concatenation $uv$, and for $k\in\N$ the $k$th
power of $u$ is $u^k$.

Consider the alphabet $A=\{0,1,\ldots,q-1\}$, and define $\rho:A^n\to A^n$ by $\rho(a_1\cdots a_n)=a_2\cdots a_n a_1$. Then the permutation group $\langle\rho\rangle$ acts on $A^n$, and
the equivalence classes under rotation are known as \emph{necklaces
of length $n$}, so called because a necklace of length $n$ can be
visualised as a regular $n$-gon where the corners represent the
``beads'', each designated one of $q$ colours. For example, if $q=2$
and $n=6$, the class of $011011$ is the necklace
$[011011]=\{011011,101101,110110\}$ and is drawn as
\begin{equation}\label{diag_necklace} \xymatrix@-1pc{ & *+[o][F-]{1} \ar@{-}[r]  & *+[o][F-]{1}\ar@{-}[dr] & \\
*+[o][F-]{0}\ar@{-}[ur] &&& *+[o][F-]{0}\ar@{-}[dl] \\
& *+[o][F-]{1}\ar@{-}[ul]  & *+[o][F-]{1}\ar@{-}[l] &}\end{equation}

The \emph{period} of a necklace $[a]$ is the smallest $d\in\N$ such
that $\rho^d(a)=a$, and it must divide the length of the necklace. A
necklace with period equal to its length is called \emph{aperiodic}
and its lexicographic least representative is known as a
\emph{Lyndon word} (originally called \emph{standard lexicographic sequences} in \cite{Lyndon}). For
example, $000011$ and $000001$ are Lyndon words of length $6$. For
every necklace $[a]$ of length $n$ and period $d$ there is a unique
subword $b$ of length $d$ such that $[b^{n/d}]=[a]$; we call $b$ the
\emph{Lyndon subword} of $[a]$. For example, the necklace $[011011]$
in \eqref{diag_necklace} has period $3$ and Lyndon subword $011$.

We identify the alphabet $A=\{0,1,\ldots,q-1\}$ with the commutative
ring $\Zmod{q}$, and let the \emph{trace} of a word be the sum of
its symbols $\pmod{q}$. We will often consider the collection
$A^n_t$ of words with length $n$ and trace $t\pmod{q}$. See
Table~\ref{table_6necklaces} for the period, trace and Lyndon
subword of each binary necklace of length $6$.

\begin{table}%[h]
\begin{center}\begin{tabular}{|c|c|c||c|c|c|}
  \hline
Necklaces of $A^6_0$ & Period & Lyndon subword & Necklaces of $A^6_1$ & Period & Lyndon subword \\
    \hline
    $[ 000000 ]$ & 1  & 0        &    $[ 000001 ]$ & 6  & 000001\\
    $[ 000011 ]$ & 6  & 000011   &    $[ 000111 ]$ & 6  & 000111\\
    $[ 000101 ]$ & 6  & 000101   &    $[ 001011 ]$ & 6  & 001011\\
    $[ 001001 ]$ & 3  & 001      &    $[ 001101 ]$ & 6  & 001101\\
    $[ 001111 ]$ & 6  & 001111   &    $[ 010101 ]$ & 2  & 01\\
    $[ 010111 ]$ & 6  & 010111   &    $[ 011111 ]$ & 6  & 011111\\
    $[ 011011 ]$ & 3  & 011      & & & \\
    $[ 111111 ]$ & 1  & 1        & & & \\
\hline
\end{tabular}\end{center}
\caption{The binary necklaces of length $6$}\label{table_6necklaces}
\end{table}

By \cite[Theorem~1.2]{Ruskey} the number of Lyndon words of length
$n$ with trace $t\pmod{q}$ over the alphabet $\{0,1,\ldots,q-1\}$ is
given by
\begin{equation}\label{eqn_L}
L_q(n,t)=\frac{1}{qn}\sum_{\substack{d|n\\
\gcd(d,q)|t}}\gcd(d,q)\mu(d)q^{n/d},
\end{equation}
where $\mu$ is the M{\" o}bius function. \end{appendix}

\end{document}